\documentclass[a4paper]{scrartcl}

\usepackage[utf8]{inputenc}
\usepackage[T1]{fontenc}

\usepackage{amsmath, amsthm, amssymb}
\usepackage{mathtools}
\usepackage{url}

\usepackage{enumitem}
\newlist{hypothenum}{enumerate}{3}
\setlist[hypothenum,1]{label=(\roman*)}

\theoremstyle{plain}

\newtheorem{theorem}{Theorem}
\newtheorem{proposition}[theorem]{Proposition}

\theoremstyle{definition}
\newtheorem{definition}[theorem]{Definition}

\theoremstyle{remark}

\newcommand{\transpose}[1]{#1^{\mathrm t}}

\newcommand{\ZZ}{\mathbb{Z}}

\newcommand{\CC}{\mathbb{C}}

\DeclareMathOperator{\Id}{Id}

\usepackage[all,cmtip]{xy}
\newcommand{\fundef}[5]{
\entrymodifiers={+!!<0pt,\fontdimen22\textfont2>}
\xymatrix@R=3pt{\llap{$#1$\;\;} {#2} \ar@{->}[r] & {#3} \\ {#4} \ar@{|->}[r] & {#5}}
} 
\def\noqed{\renewcommand{\qedsymbol}{}}

\newcommand{\eg}{e.g.\ }

\newcommand{\PolyHR}[1]{P_{\mathrm{HR}}^{#1}}

\begin{document}

\title{New sequences of non-free rational points}
\author{Ilia Smilga \thanks{The author is supported by the European Research Council (ERC) under the European Union Horizon 2020 research and innovation programme (ERC starting grant DiGGeS, grant agreement No. 715982).}}
              
\maketitle

\begin{abstract}
We exhibit some new infinite families of rational values of~$\tau$, some of them squares of rationals, for which the group or even the semigroup generated by the matrices $\left( \begin{smallmatrix} 1 & 1 \\ 0 & 1 \end{smallmatrix} \right)$ and~$\left( \begin{smallmatrix} 1 & 0 \\ \tau & 1 \end{smallmatrix} \right)$ is not free.
\end{abstract}


For each $\alpha, \beta \in \CC$, we define the group
\[\Gamma(\alpha, \beta) := \left\langle \begin{pmatrix} 1 & \alpha \\ 0 & 1 \end{pmatrix}, \begin{pmatrix} 1 & 0 \\ \beta & 1 \end{pmatrix} \right\rangle,\]
and we define $S(\alpha, \beta)$ to be the semigroup generated by the same two matrices. The problem of determining values of the product $\tau := \alpha \beta$ for which the group $\Gamma(\alpha, \beta)$ (conjugate by $\left( \begin{smallmatrix} 1 & 0 \\ 0 & \alpha \end{smallmatrix} \right)$ to $\Gamma(1, \tau)$) is free has already attracted a considerable amount of attention \cite{San47, Bre55, CJR58, Ree61, LU69, KS94, Gil08}; the corresponding problem for the semigroup $S(\alpha, \beta)$ has also been studied \cite{BC78, Sla15}. In particular, a lot of people have focused on the special case where $\tau$ is the square of a rational \cite{LU69, Ign86, Bea93, TT96} (which is a natural condition as it corresponds to studying $\Gamma(\mu, \mu)$ with rational $\mu$), or more generally any rational number \cite{Bam00, KK}.

To avoid confusion, we should warn the reader that previous authors have variously used as the key parameter either the product $\alpha \beta$ (\eg \cite{KS94, KK}), its half, usually denoted by~$\lambda$ (focusing on $\Gamma(2, \lambda)$) (\eg \cite{CJR58, Ree61, LU69, Bam00, Gil08, Sla15}), or its square root, usually denoted by~$\mu$ (\eg \cite{San47, Bre55, BC78, Bea93, TT96}). We make the former choice, hence the following terminology:

\begin{definition}
We say that a number $\tau \in \CC$ is \emph{free} (resp. \emph{semigroup-free}) if the group $\Gamma(1, \tau)$ (resp. the semigroup $S(1, \tau)$) is free.
\end{definition}

An elementary observation (first made by Brenner \cite{Bre55}) is that for $\tau \geq 4$, the group $\Gamma(1, \tau)$ is always Schottky, hence free. So for groups, it suffices to study the interval $(-4, 4)$ (in fact even $(0, 4)$, given the symmetry $\Gamma(1, \tau) = \Gamma(1, -\tau)$). Similarly, for $\tau \geq 1$, the semigroup $S(1, \tau)$ is Schottky (see also \cite[Theorem 2.6]{BC78}), so it suffices to study the interval $(-4, 1)$.

Beyond this, very little is known. It is conjectured that all rational $\tau \in (-4, 4)$ are non-free, and that all rational $\tau \in (-4, 1)$ are non-semigroup-free. There are no known counterexamples, but only a limited number of values for which these two conjectures have been proved. Indeed, group relations are easiest to find for $\tau$ close to~$0$. A spectacular development in this direction is \cite{KK}, proving that all rational numbers $\tau \in (-4, 4)$ whose numerator is at most $23$, or between $25$ and~$27$, are non-free. On the other hand, as $\tau$ gets closer to~$4$, finding relations becomes extremely hard.

It is even harder to find accumulation points of non-free rational values. It seems that the only ones known so far (besides $0$, which is trivial as $\frac{1}{n}$ is never free) are those of the form $\frac{1}{n}$ \cite{Bea93} and those of the form $\frac{n}{n \pm 1}$ \cite{TT96}. As for accumulation points of non-semigroup-free rational values, only $0$ is known.

Our goal in this paper has been to find accumulation points of non-free and non-semigroup-free rational values (or squares of rationals) that lie as close as possible to the endpoints of the corresponding interval. The result that we obtain can be summarised as follows.

\begin{theorem}
\label{accu_points_list}
The set of~$\tau$ that are rational and non-semigroup-free (hence in particular non-free) has $-2 \pm \sqrt{2}$, $-\phi^{\pm 2}$, $-2$ and~$1$ as accumulation points (where $\phi = \frac{\sqrt{5}+1}{2}$ is the golden ratio). The set of~$\tau$ that are squares of rationals and non-semigroup-free (hence in particular non-free) has $1$ as an accumulation point.
\end{theorem}
\begin{proof}
This follows from the more precise Theorem~\ref{infinite_families_description} below, which gives a concrete description of the infinite families that accumulate at these points. Note that:
\begin{itemize}
\item the values given in \ref{itm:golden} for positive (resp. negative) $k$ are in fact the convergents of the continued fraction of $\phi^2$ (resp. of $\phi^{-2}$);
\item the values given in \ref{itm:2_sqrt2} for positive (resp. negative) $k$ are in fact the convergents of the continued fraction of $2 + \sqrt{2}$ (resp. of $2 - \sqrt{2}$). \qedhere
\end{itemize}
\end{proof}

The technique used to prove points \ref{itm:golden} and \ref{itm:2_sqrt2} is similar to that used by Beardon \cite{Bea93} and Tan and Tan \cite{TT96}: a one-parameter family of relations leading to values of $\tau$ that are convergents of some continued fraction. The main difference is that these two papers focused on the case where $\tau$ is the square of a rational, whereas we consider all rational $\tau$. Moreover, with similar methods, one can also construct many other accumulation points; we have chosen to present only those that are closest to the endpoints of the interval.

For each member of each of the following five families, we will actually exhibit an explicit relation for the corresponding group or semigroup (see Proposition~\ref{relations_description}).

\begin{theorem}
\label{infinite_families_description}
All of the following values of~$\tau$ are non-free. The values given in~\ref{itm:Markov_like} are also non-semigroup-free. For $\tau$ as in \ref{itm:2kp1_k} with $k > 0$ or as in \ref{itm:golden} or~\ref{itm:2_sqrt2} with $(-1)^k k > 0$, its negative $-\tau$ is non-semigroup-free.
\begin{hypothenum}
\item \label{itm:2kp1_2k} $\tau = \left( \frac{2k \pm 1}{2k} \right)^2$, for all integer $k \neq 0$.
\item \label{itm:Markov_like} $\tau = \left( \frac{n-1}{n} \right)^2$, for all $n = \frac{6}{\sigma_0 \sigma_1} u^\sigma_k u^\sigma_{k+1}$, where $\sigma = (\sigma_0, \sigma_1)$ is any pair of distinct numbers among $\{1, 2, 3\}$, and $(u_k^\sigma)_{k \in \ZZ}$ is the integer sequence determined by the following recurrence relations:
\begin{equation}
\label{eq:u_sequence_def}
\begin{cases}
u_0^\sigma = u_1^\sigma = 1; \\
\forall k \in \ZZ,\quad u_{k-1}^\sigma - 2\sigma_{(k \operatorname{mod} 2)} u_{k}^\sigma + u_{k+1}^\sigma = 0
\end{cases}
\end{equation}
(see discussion below for explicit list of values).
\item \label{itm:2kp1_k} $\tau = \frac{2k \pm 1}{k}$, for all integer $k \neq 0$.
\item \label{itm:golden} $\tau = \frac{F_{k+2}}{F_k}$ for all integer $k \neq 0$, where $F$ is the Fibonacci sequence ($F_1 = F_2 = 1$). Explicitly, these are
\begin{equation}
2, 3, \frac{5}{2}, \frac{8}{3}, \frac{13}{5}, \frac{21}{8}, \ldots
\end{equation}
and their reciprocals.
\item \label{itm:2_sqrt2} $\tau = \frac{H_{k+1}}{P_k}$ for all integer $k \neq 0$, where $P_k$ are the Pell numbers and $H_k$ are the half-companion Pell numbers, given by $\left( \begin{smallmatrix} H_k \\ P_k \end{smallmatrix} \right) := \left( \begin{smallmatrix} 1 & 2 \\ 1 & 1 \end{smallmatrix} \right)^k \left( \begin{smallmatrix} 1 \\ 0 \end{smallmatrix} \right)$. Explicitly, these are
\begin{equation}
3, \frac{7}{2}, \frac{17}{5}, \frac{41}{12}, \frac{99}{29}, \frac{239}{70}, \ldots
\end{equation}
and twice their reciprocals.
\end{hypothenum}
\end{theorem}
\begin{proof}[Discussion of the values given in \ref{itm:Markov_like}.]
One readily checks the following symmetry of the sequences $u_k^\sigma$:
\begin{equation}
\forall \sigma_0 \neq \sigma_1 \in \{1, 2, 3\},\; \forall k \in \ZZ,\quad u^{(\sigma_0, \sigma_1)}_k = u^{(\sigma_1, \sigma_0)}_{1-k},
\end{equation}
so it suffices to consider only three of the six sequences. Moreover the sequences $u^{(1,2)}$ and $u^{(1,3)}$ are symmetric:
\begin{equation}
\forall \sigma_1 \in \{2, 3\},\; \forall k \in \ZZ,\quad u^{(1, \sigma_1)}_k = u^{(1, \sigma_1)}_{-k}
\end{equation}
(indeed observe that in this case $u^\sigma_{-1} = 2 u^\sigma_0 - u^\sigma_1 = 1 = u^\sigma_1$, then use induction), so it suffices to consider their nonnegative terms. Explicitly, we have:
\begin{gather}
\left(u^{(1,2)}_{\pm k}\right)_{k \geq 0} = 1, 1, 3, 5, 17, 29, 99, 169, \ldots; \\
\left(u^{(1,3)}_{\pm k}\right)_{k \geq 0} = 1, 1, 5, 9, 49, 89, 485, 881, \ldots; \\
\left(u^{(2,3)}_k\right)_{k \in \ZZ} = \ldots, 1427, 373, 65, 17, 3, 1, 1, 5, 19, 109, 417, 2393, \ldots.
\end{gather}
These give rise to the following values of~$n$:
\begin{gather}
\left(3 u^{(1,2)}_k u^{(1,2)}_{k+1}\right)_{k \geq 0} = 3, 9, 45, 255, 1479, 8613, 50193, \ldots; \\
\left(2 u^{(1,3)}_k u^{(1,3)}_{k+1}\right)_{k \geq 0} = 2, 10, 90, 882, 8722, 86330, \ldots; \\
\left(u^{(2,3)}_k u^{(2,3)}_{k+1}\right)_{k \in \ZZ} = \ldots, 24245, 1105, 51, 3, 1, 5, 85, 2071, 45453, \ldots. 
\end{gather} \noqed
\end{proof}

In fact, instead of relations, we will exhibit so-called ``half-relations'', that we will define right now. We will then see (Proposition~\ref{half_rels_vs_rels}) that each half-relation easily allows one to write down an actual relation.

\begin{definition}
Let $\tau \in \CC$, and let us fix the notations $g := \left( \begin{smallmatrix} 1 & 1 \\ 0 & 1 \end{smallmatrix} \right)$ and $h_\tau := \left( \begin{smallmatrix} 1 & 0 \\ \tau & 1 \end{smallmatrix} \right)$.
We say that a sequence $(a_1, \ldots, a_l) \in \ZZ^l$ is a \emph{half-relation} for~$\tau$ if:
\begin{hypothenum}
\item for odd~$l$, the matrix $M = g^{a_1} h_\tau^{a_2} \cdots h_\tau^{a_{l-1}} g^{a_l}$ satisfies
\begin{equation}
\label{eq:odd_half_rel_def}
\tau c_{12}(M) - c_{21}(M) = 0;
\end{equation}
\item for even~$l$, the matrix $M = g^{a_1} h_\tau^{a_2} \cdots g^{a_{l-1}} h_\tau^{a_l}$ satisfies
\begin{equation}
\label{eq:even_half_rel_def}
c_{11}(M) - c_{22}(M) = 0,
\end{equation}
\end{hypothenum}
where $c_{ij}$ denotes the $(i, j)$-th coefficient, so that
\begin{equation}
\label{eq:cij_def}
\forall M,\quad M =: \begin{pmatrix} c_{11}(M) & c_{12}(M) \\ c_{21}(M) & c_{22}(M) \end{pmatrix}.
\end{equation}
\end{definition}

The point of half-relations is that they allow us to construct a special kind of relations for the group or semigroup corresponding to $\pm \tau$:

\begin{proposition}
\label{half_rels_vs_rels}
A sequence $(a_1, \ldots, a_l) \in \ZZ^l$ is a half-relation for~$\tau$ if and only if the identity
\begin{equation}
\label{eq:symmetric_relation}
g^{a_1} h_\tau^{a_2} \cdots f^{a_l} = h_\tau^{a_l} g^{a_{l-1}} \cdots f^{a_1},\qquad \text{ where } f :=
\begin{cases}
g & \text{ if $l$ is odd} \\
h_\tau & \text{ if $l$ is even,}
\end{cases}
\end{equation}
or equivalently
\begin{equation}
\label{eq:one_sided_relation}
g^{a_1} h_\tau^{a_2} \cdots g^{-a_2} h_\tau^{-a_1} = \Id,
\end{equation}
holds. Such an identity provides a nontrivial relation:
\begin{itemize}
\item for the group $\Gamma(1, \tau)$ if $\forall i,\; a_i \neq 0$;
\item for the semigroup $S(1, \tau)$ if $\forall i,\; a_i > 0$;
\item for the semigroup $S(1, -\tau)$ if $\forall i,\; (-1)^i a_i > 0$.
\end{itemize}
\end{proposition}

\begin{proof}
Let $M = g^{a_1} h_\tau^{a_2} \cdots$ denote the left-hand side of~\eqref{eq:symmetric_relation}. Then $(a_1, \ldots, a_l)$ is a half-relation for~$\tau$ if and only if $M$ is a fixed point of:
\begin{itemize}
\item the involution $\left( \begin{smallmatrix} a & b \\ c & d \end{smallmatrix} \right) \mapsto \left( \begin{smallmatrix} a & \smash{c/\tau} \\ \tau b & d \end{smallmatrix} \right)$, or equivalently $M \mapsto \left( \begin{smallmatrix} 1 & 0 \\ 0 & \tau \end{smallmatrix} \right) \transpose{M} \left( \begin{smallmatrix} 1 & 0 \\ 0 & \tau \end{smallmatrix} \right)^{-1}$, for odd~$l$;
\item the involution $\left( \begin{smallmatrix} a & b \\ c & d \end{smallmatrix} \right) \mapsto \left( \begin{smallmatrix} d & b \\ c & a \end{smallmatrix} \right)$, or equivalently $M \mapsto \left( \begin{smallmatrix} 1 & 0 \\ 0 & -1 \end{smallmatrix} \right) M^{-1} \left( \begin{smallmatrix} 1 & 0 \\ 0 & -1 \end{smallmatrix} \right)^{-1}$, for even~$l$.
\end{itemize}
Both of them are antimorphisms; moreover the former switches $g$ and~$h_\tau$, whereas the latter fixes both of them. In both cases, the involution maps both sides of~\eqref{eq:symmetric_relation} to each other.

Now assume that this is true. Then \eqref{eq:symmetric_relation}, usually rewritten as \eqref{eq:one_sided_relation} for groups, is a relation in $\Gamma(1, \tau)$, which is nontrivial if the coefficients are all nonzero. If they are all positive, then \eqref{eq:symmetric_relation} is in fact a relation in $S(1, \tau)$. If they have alternating signs (note that whenever $(a_1, \ldots, a_l)$ is a half-relation, so is $(-a_1, \ldots, -a_l)$), then \eqref{eq:symmetric_relation} (for even~$l$) or \eqref{eq:one_sided_relation} (for odd~$l$) is a relation in $S(1, -\tau)$.
\end{proof}

Let us also observe that being a half-relation is a polynomial condition: more precisely, the left-hand side of \eqref{eq:odd_half_rel_def} (for odd~$l$) or \eqref{eq:even_half_rel_def} (for even~$l$) is a polynomial in $a_1, \ldots, a_l$ and~$\tau$. One also easily checks by induction that this polynomial is always divisible by~$\tau$, so we can factor $\tau$ out:
\begin{definition}
For each $l \geq 0$, we define the polynomial 
\begin{equation}
\PolyHR{l}(a_1, \ldots, a_l; \tau) :=
\begin{cases}
(c_{12} - \frac{1}{\tau}c_{21})(g^{a_1} h_\tau^{a_2} \cdots g^{a_l}) &\text{ if $l$ is odd;} \\
\frac{1}{\tau}(c_{11} - c_{22})(g^{a_1} h_\tau^{a_2} \cdots h_\tau^{a_l}) &\text{ if $l$ is even}
\end{cases}
\end{equation}
(where $c_{ij}$ is the $i,j$-th coefficient, as in \eqref{eq:cij_def}).
\end{definition}

Thus, by construction, whenever $\tau \in \CC$ is a root of some polynomial $\PolyHR{l}(a_1, \ldots, a_l)$ for some tuple $(a_1, \ldots, a_l)$ of nonzero (resp. positive) integers, it is non-free (resp. non-semigroup-free).
Explicitly, for small~$l$, these polynomials are as follows:
\begin{align*}
\PolyHR{1}(a_1; \tau) &= a_1 \\
\PolyHR{2}(a_1, a_2; \tau) &= a_1 a_2 \\
\PolyHR{3}(a_1, a_2, a_3; \tau) &= a_1 a_2 a_3 \tau + a_1 - a_2 + a_3 \\
\PolyHR{4}(a_1, \ldots, a_4; \tau) &= a_1 a_2 a_3 a_4 \tau + a_1 a_2 - a_2 a_3 + a_3 a_4 + a_1 a_4 \\
\begin{split}
\PolyHR{5}(a_1, \ldots, a_5; \tau) &= a_1 a_2 a_3 a_4 a_5 \tau^2 + (a_1 a_2 a_3 - a_2 a_3 a_4 + a_1 a_2 a_5 + a_1 a_4 a_5 + a_3 a_4 a_5)\tau \\
&\qquad + a_1 - a_2 + a_3 - a_4 + a_5.
\end{split}
\end{align*}
The observation that non-free values can be obtained as roots of certain polynomials has already been made previously. Even more specifically, when $l = 2k+1$ is odd, the polynomial $\PolyHR{2k+1}$ coincides (up to the substitution $\tau = - 2\lambda$ and a sign switch for the odd $a_i$'s) with Bamberg's polynomial $B_k$ \cite{Bam00}.

We are now ready to exhibit the half-relations that prove all the statements of Theorem~\ref{infinite_families_description}. Note that in point~\ref{itm:2kp1_k}, the subcase \ref{itm:2kp1_k_general} suffices by itself to prove the Theorem in the general case; but we have also presented two additional special cases, \ref{itm:2kp1_k_even} and \ref{itm:2kp1_k_quad}, for which we can shorten the half-relation by one coefficient.

\begin{proposition}
\label{relations_description}
Fix some $k \in \ZZ$, assumed to be nonzero except in point~\ref{itm:Markov_like}, and some pair $\sigma = (\sigma_0, \sigma_1)$ of distinct coefficients among $\{1, 2, 3\}$.
\begin{hypothenum}
\item[\ref{itm:2kp1_2k}] The value $\tau = \left(\frac{2k-1}{2k}\right)^2$ has $(1, -1, -k, k(4k+4))$ as a half-relation.
\item[\ref{itm:Markov_like}] The value $\tau = \left(\frac{n-1}{n}\right)^2$, where $n = \frac{6}{\sigma_0 \sigma_1} u^\sigma_k u^\sigma_{k+1}$, has $(1, \frac{6}{\sigma_{k+1}} (u^\sigma_k)^2, \frac{6}{\sigma_k} (u^\sigma_{k+1})^2, 1)$ as a half-relation (where $\sigma_k$ is a notation shortcut for $\sigma_{(k \operatorname{mod} 2)}$).
\item[\ref{itm:2kp1_k}] The value $\tau = \frac{2k + 1}{k}$ has following half-relations:
\begin{enumerate}[label=\alph*)]
\item \label{itm:2kp1_k_general} All sequences $(k, -1, 1, -1, k, x)$, where $x$ is any integer.
\item \label{itm:2kp1_k_even} If $k = 2t$ for some integer $t$, the sequence $(1, -1, 1, -t, -4t^2+2t-2)$.
\item \label{itm:2kp1_k_quad} If $k = \frac{t(t+1)}{2}-1$ for some integer $t$, the sequence  $(1, -1, 1, -t+1, -t-2)$.
\end{enumerate}
\item[\ref{itm:golden}] The value $\tau = \frac{F_{k+2}}{F_k}$ has $(1, -1, 1, -1, 2(-1)^kF_{k-1} F_{k})$ as a half-relation (where $F$ is the Fibonacci sequence, $F_1 = F_2 = 1$).
\item[\ref{itm:2_sqrt2}] The value $\tau = \frac{H_{k+1}}{P_k}$ has $((-1)^kP_{k-1}P_k, -1, 1, -1, 1, -1, 1, -1, (-1)^kP_{k-1}P_k, x)$ as a half-relation for any integer~$x$ (where $P_k$ are the Pell numbers and $H_k$ are the half-companion Pell numbers, given by $\left( \begin{smallmatrix} H_k \\ P_k \end{smallmatrix} \right) := \left( \begin{smallmatrix} 1 & 2 \\ 1 & 1 \end{smallmatrix} \right)^k \left( \begin{smallmatrix} 1 \\ 0 \end{smallmatrix} \right)$).
\end{hypothenum}
\end{proposition}
A remark about point~\ref{itm:Markov_like}: in fact, one can show that these are the \emph{only} values of $n$ for which $\tau = \left(\frac{n-1}{n}\right)^2$ has a half-relation of length~$4$ with positive coefficients.
\begin{proof}
To prove these statements, we check that all of the listed values make the polynomial $\PolyHR{}$ vanish.
\begin{itemize}
\item For points \ref{itm:2kp1_2k} and~\ref{itm:2kp1_k}, this is a straightforward computation.
\item For point~\ref{itm:golden}, we start by computing
\begin{equation}
\PolyHR{5}(1, -1, 1, -1, N; \tau) = N \tau^2 - (3 N + 2) \tau + (N + 4).
\end{equation}
It remains to verify that for all $k \neq 0$, the values $\tau = \frac{F_{k+2}}{F_{k}}$ and $N = 2(-1)^kF_{k-1}F_{k}$ make this expression vanish. This is straightforward (if tedious), for example by plugging in the closed formula
\[F_k = \frac{\phi^k - (-\phi)^{-k}}{\phi + \phi^{-1}}\]
(where $\phi := \frac{1+\sqrt{5}}{2}$ is the golden ratio) and expanding.
\item For point~\ref{itm:2_sqrt2}, we start by computing
\begin{align}
&\PolyHR{10}(N, -1, 1, -1, 1, -1, 1, -1, N, x; \tau) = \nonumber \\
&= x \Big( N^2 \tau^4 - (6 N^2 + 2 N) \tau^3 + (10 N^2 + 10 N + 1) \tau^2 - (4 N^2 + 12 N + 4) \tau + 2 N + 3 \Big) \nonumber \\
&= x \Big( N \tau^2 - (2 N + 1) \tau + 1 \Big) \Big( N \tau^2 - (4 N + 1) \tau + (2 N + 3) \Big).
\end{align}
It turns out that for all $k \neq 0$, the values $\tau = \frac{H_{k+1}}{P_k}$ and $N = (-1)^kP_{k-1}P_k$ make the last factor vanish (so that the whole expression vanishes regardless of the value of~$x$). This is straightforward (if tedious) to verify, for example by plugging in the closed formulas
\begin{equation}
P_k = \frac{\alpha^k - (-\alpha)^{-k}}{\alpha + \alpha^{-1}},\qquad H_k = \frac{\alpha^k + (-\alpha)^{-k}}{\alpha - \alpha^{-1}}
\end{equation}
(where $\alpha = 1 + \sqrt{2}$) and expanding.
\item Finally, for point~\ref{itm:Markov_like}, we start by computing
\begin{align}
\label{eq:PHR_for_u_seq}
&\PolyHR{4}\left(1, \frac{6}{\sigma_{k+1}} (u^\sigma_k)^2, \frac{6}{\sigma_k} (u^\sigma_{k+1})^2, 1; \left(1 - \frac{1}{\frac{6}{\sigma_0 \sigma_1} u^\sigma_k u^\sigma_{k+1}}\right)^2\right) = \nonumber \\
&\qquad\qquad\qquad\qquad\qquad\qquad\qquad\qquad= P^{(\sigma_k, \sigma_{k+1})}(u^\sigma_k, u^\sigma_{k+1}) \nonumber \\
&\qquad\qquad\qquad\qquad\qquad\qquad\qquad\qquad= \begin{cases}
P^\sigma (u^\sigma_k, u^\sigma_{k+1}) &\text{ if $k$ is even} \\
P^\sigma (u^\sigma_{k+1}, u^\sigma_k) &\text{ if $k$ is odd,}
\end{cases}\
\end{align}
where, for all $\sigma = (\sigma_0, \sigma_1)$, $P^\sigma$ is the polynomial given by
\begin{equation}
P^\sigma(x, y) := 1 + \sigma_0 \sigma_1 + \frac{6}{\sigma_1}x^2 + \frac{6}{\sigma_0} y^2 - 12 x y.
\end{equation}
Recalling the definition~\eqref{eq:u_sequence_def} of the sequences~$u^\sigma$, to prove (by induction on~$k$) that the right-hand side of \eqref{eq:PHR_for_u_seq} vanishes for all $k \in \ZZ$, it suffices to check the following two identities. On the one hand, we have
\begin{equation}
\forall \sigma_0 \neq \sigma_1 \in \{1, 2, 3\},\quad P^\sigma(1, 1) = 0
\end{equation}
(indeed the triplet $(\sigma_0 \sigma_1, \frac{6}{\sigma_0}, \frac{6}{\sigma_1})$ is always some permutation of $(2, 3, 6)$). On the other hand, we easily check that
\begin{equation}
\forall \sigma, \forall x, y,\quad P^\sigma(x, y) = P^\sigma(2\sigma_1 y - x, y) = P^\sigma(x, 2\sigma_0 x - y). \qedhere
\end{equation}
\end{itemize}
\end{proof}

\begin{proof}[Proof of Theorem~\ref{infinite_families_description}]
Theorem~\ref{infinite_families_description} now easily follows by Proposition~\ref{half_rels_vs_rels}: indeed, one readily checks that the coefficients of almost all the half-relations listed above are nonzero, and have the required signs. There are only three exceptions:
\begin{itemize}
\item $k = 1$ in cases~\ref{itm:golden} and~\ref{itm:2_sqrt2}, which yields respectively $\tau = 2$ and $\tau = 3$, which are well-known to be non-free. Also in these cases, one can still write down the relation \eqref{eq:symmetric_relation} (or \eqref{eq:one_sided_relation}); some cancellations occur, but one still recovers the nontrivial relation $g^2 h_\tau^{-1} g h_\tau^{-2} g h_\tau^{-1} = \Id$ (for $\tau = 2$) and $g h_\tau^{-1} g h_\tau^{-1} g h_\tau^{-1} = \Id$ (for $\tau = 3$).
\item $k = -1$ in case~\ref{itm:2kp1_2k}, so that $\tau = \frac{9}{4}$. Then the half-relation written above leads to a trivial relation, so we need to find something else. To wit, $(1, -1, 1, 14, 2)$ is a half-relation with nonzero coefficients for $\tau = \frac{9}{4}$. \qedhere
\end{itemize}
\end{proof}

\section*{Acknowledgements}

I would like to thank Gregory Margulis for drawing my attention to this problem, and Paul Mercat for some helpful discussions.

\bibliographystyle{alpha}
\bibliography{/home/ilia/Documents/Travaux_mathematiques/mybibliography.bib}
\end{document}